\documentclass[12pt]{amsart}

\usepackage[margin=2.75cm]{geometry}

\usepackage{amsmath}
\usepackage{amssymb}
\usepackage{amsthm}
\usepackage{epsfig}
\usepackage{url}
\usepackage{array}
\usepackage{varwidth}
\usepackage{tikz}
\usetikzlibrary{patterns}
\usepackage{multirow}
\usepackage{textcomp}
\usepackage{bbm}
\usepackage{float}
\usepackage{setspace}
\usepackage{array}
\usepackage{color}
\usepackage{ifthen}
\usepackage{enumerate}

\newtheorem{theorem}{Theorem}
\newtheorem{proposition}[theorem]{Proposition}

\newtheorem{corollary}[theorem]{Corollary}

\newtheorem{lemma}[theorem]{Lemma}

\newtheorem*{mattreethm}{Matrix Tree Theorem}

\theoremstyle{definition}

\theoremstyle{remark}

\newtheorem{remark}[theorem]{Remark}

\numberwithin{equation}{section}

\newcommand{\vu}{\mathbf{u}}
\newcommand{\vv}{\mathbf{v}}

\newcommand{\RR}{\mathbb{R}}

\newcommand{\one}{\mathbf{1}}

\begin{document}

\title[]{Eigenvalues of graph Laplacians via rank-one perturbations}

\author{Steven Klee}
\address{Seattle University, Department of Mathematics, 901 12th Avenue, Seattle, WA 98122, USA}
\email{klees@seattleu.edu}
\thanks{Steven Klee's research was supported by NSF grant DMS-1600048.}
\author{Matthew T. Stamps}
\address{Yale-NUS College, Division of Science, 12 College Avenue West, Singapore 138610}
\email{matt.stamps@yale-nus.edu.sg}



\begin{abstract}
We show how the spectrum of a graph Laplacian changes with respect to a certain type of rank-one perturbation.  We apply our finding to give new short proofs of the spectral version of Kirchhoff's Matrix Tree Theorem and known derivations for the characteristic polynomials of the Laplacians for several well known families of graphs, including complete, complete multipartite, and threshold graphs.
\end{abstract}

\maketitle

%

\section{Introduction}

In this paper, we study finite simple graphs, i.e., graphs with finite vertex sets that do not contain loops or multiple edges.  We use $V(G)$ and $E(G)$ to denote the vertex set and edge set of a graph $G$, respectively, and we take $V(G) = [n] := \{1,\ldots,n\}$ unless stated otherwise.  Recall that a \textit{spanning tree} in a graph $G$ is a subgraph $T \subseteq G$ such that 
\begin{enumerate}
\item $V(T) = V(G)$, 
\item $T$ is connected, and 
\item $T$ does not contain any cycles.
\end{enumerate}
We are interested in the number $\tau(G)$ of spanning trees in $G$. 

Computing $\tau(G)$ for an arbitrary graph $G$ on $n$ vertices can be done reasonably efficiently\footnote{We will see that the number of spanning trees can be computed as the determinant of an $(n-1) \times (n-1)$ matrix, which can be done by na\"ive row reduction in $O(n^3)$ time.} with the help of its \textit{Laplacian matrix}, $L(G)$, which is the $n \times n$ matrix with entries 
$$
L(G)(i,j) = 
\begin{cases}
\deg(i) & \text{ if } i = j, \\
-1 & \text{ if } i \neq j \text{ and } \{i,j\} \in E(G), \\
0 & \text{otherwise.}
\end{cases}
$$
Note that the rows (and columns) of $L(G)$ sum to zero, meaning the all ones vector in $\RR^n$, which we denote by $\one_n$, always lies in the nullspace of $L(G)$. Thus, we must account for the nullity of $L(G)$ when attempting to make any connection between $G$ and its Laplacian. For any $1 \leq i,j \leq n$, not necessarily distinct, let $L(G)_{i,j}$ be the matrix obtained by eliminating the $i$th row and $j$th column from $L(G)$.  A celebrated result of Kirchhoff \cite{Kirchhoff} reveals a fundamental connection between the Laplacian matrix and the number of spanning trees in a graph.

\begin{mattreethm}[\cite{Kirchhoff}]
Let $G$ be a graph on $n$ vertices.  
\begin{enumerate}[(i)]
\item For any vertices $i$ and $j$, not necessarily distinct, 
$$
\tau(G) = (-1)^{i+j} \det(L(G)_{i,j}).
$$
\item If $\lambda_1,\ldots,\lambda_n$ are the eigenvalues of $L(G)$ with $\lambda_n = 0$, then 
$$
\tau(G) = \cfrac{\lambda_1 \cdot \ldots \cdot \lambda_{n-1}}{n}.
$$
\end{enumerate}
\end{mattreethm}

The Matrix Tree Theorem is beautiful in its simplicity, but the requirement that one must choose a row and column to eliminate in part (i) is somewhat unsatisfying since some choices may be more convenient than others.  We demonstrated a technique for easily counting spanning trees in many well-studied families of graphs by adding rank-one matrices to their Laplacian matrices in \cite{Klee-Stamps-unweighted}. 

\begin{lemma}[\cite{Klee-Stamps-unweighted}]
\label{KS-rank-one-result}
Let $G$ be a graph on $n$ vertices with Laplacian matrix $L$, and let $\vu = (u_i)_{i \in [n]}$ and $\vv = (v_i)_{i \in [n]}$ be column vectors in $\RR^n$.  Then 
\begin{equation}\label{eqn:KS-unweighted}
\det(L + \vu\vv^T) = \left( \sum_{i = 1}^n u_i \right) \left( \sum_{i = 1}^n v_i\right) \tau(G).
\end{equation}
\end{lemma}

This approach not only allows one to work with the full Laplacian matrix, but it transfers the choice of which row and column to eliminate when computing the determinant to a choice of which rank-one matrix to add.  In many cases, this leads to simpler computations --- for example, if $K_n$ is the complete graph on $n$ vertices and $\vu = \vv = \one_n$, then $L(K_n) + \vu\vv^T = nI_n$, whose determinant is clearly $n^n$.  Cayley's  formula \cite{Cayley}, $\tau(K_n) = n^{n-2}$, follows immediately from Equation~\eqref{eqn:KS-unweighted}.

A limitation of Lemma \ref{KS-rank-one-result}, however, is that it only gives spanning tree counts; it does not, for instance, allow one to glean stronger information about the Laplacian eigenvalues.  The purpose of this paper is to present an analogous result to Lemma~\ref{KS-rank-one-result} (see Theorem~\ref{thm:main}) for Laplacian eigenvalues and to demonstrate its applicability to several well known families of graphs.

Typically, the proof of part (ii) of the Matrix Tree Theorem requires some analysis relating the characteristic polynomial of $L(G)_{i,j}$ to that of $L(G)$. Instead, we present a direct proof as a consequence of our main theorem, which shows how the characteristic polynomial of a graph Laplacian matrix changes when we add a certain type of rank-one matrix.

\begin{theorem} \label{thm:main}
Let $G$ be a graph on $n$ vertices with Laplacian matrix $L = L(G)$, and let $\lambda_1,\ldots,\lambda_n$ be the eigenvalues of $L$ with $\lambda_n = 0$.  

\begin{enumerate}[(i)]
    \item There exist orthogonal eigenvectors $\vv_1,\ldots,\vv_n$ with $L\vv_i = \lambda_i \vv_i$ for all $i \in [n]$ and $\vv_n = \one_n$. 
    \item Let $\vu = (u_i)_{i \in [n]}$ be an arbitrary column vector in $\RR^n$. The characteristic polynomial of the matrix $\overline{L} := L + \vu\one_n^T$ is 
\begin{align}
\label{eqn:u-update} 
\det(\overline{L} - \lambda I_n ) &= (\lambda_1 - \lambda)\cdots(\lambda_{n-1}-\lambda)\left(\sum_{i=1}^n u_i - \lambda\right). 
\end{align}
\end{enumerate}
\end{theorem}

We prove Theorem \ref{thm:main} in Section~\ref{section:mainproof}.  As an immediate consequence, however, we obtain a short proof of part (ii) of the Matrix Tree Theorem.

\begin{corollary}
Let $G$ be a graph on $n$ vertices.  If $\lambda_1,\ldots,\lambda_n$ are the eigenvalues of $L(G)$ with $\lambda_n = 0$, then 
$$
\tau(G) = \cfrac{\lambda_1 \cdot \ldots \cdot \lambda_{n-1}}{n}.
$$
\end{corollary}

\begin{proof}
By taking $\vu = \one_n$ in Theorem \ref{thm:main}, it follows that the eigenvalues of $L + \one_n\one_n^T$ are $\lambda_1,\ldots,\lambda_{n-1}, n$.   Thus, by Equation~\eqref{eqn:KS-unweighted}, 
$$
n^2 \tau(G) = \det(L + \one_n\one_n^T) = \lambda_1 \cdot \ldots \cdot \lambda_{n-1} \cdot n.
$$
\end{proof}

The rest of the paper is structured as follows: In Section~\ref{section:mainproof}, we prove our main result, Theorem~\ref{thm:main}.  In Section~\ref{section:applications}, we demonstrate how Theorem~\ref{thm:main} can be applied to derive the characteristic polynomials for Laplacians of several well known families of graphs, including complete graphs, complete multipartite graphs, complete bipartite graphs with perfect matchings removed, and threshold graphs.

\section{Proof of the Main Result} \label{section:mainproof}

\begin{proof}[Proof of Theorem \ref{thm:main}.]
For part (i), we note that since $L$ is a real symmetric matrix, the Spectral Theorem \cite{Halmos} implies $\RR^n$ has an orthogonal basis  of eigenvectors of $L$.  Moreover, because this orthogonal basis of eigenvectors can be constructed inductively beginning with an arbitrary basis of eigenvectors, we may assume $\vv_n = \one_n$ with corresponding eigenvalue $\lambda_n = 0$.

For part (ii), we begin with the observation that, for all $i<n$, 
$$
\left( L + \vu\one_n^T\right) \vv_i = L \vv_i + \vu \left( \one_n^T \vv_i\right) = \lambda_i \vv_i,
$$
since $\vv_i$ is orthogonal to $\vv_n = \one_n$.  This means $\vv_1,\ldots,\vv_{n-1}$ are also eigenvectors for $\overline{L}$ with corresponding eigenvalues $\lambda_1,\ldots,\lambda_{n-1}$. 

Next, we use the fact that $\overline{L}$ and its transpose have the same eigenvalues to see that
\begin{align*}
\left(L + \vu\one_n^T\right)^T\one_n &= \left(L^T + \one_n\vu^T\right) \one_n \\
&\stackrel{(*)}{=} L\one_n + \one_n \left( \vu^T \one_n \right) \\
&= \left(\sum_{i=1}^n u_i\right) \one_n.
\end{align*}
Note that $(*)$ follows from the fact that $L$ is symmetric.  Thus, $\sum u_i$ is an eigenvalue for $\overline{L}$ with corresponding (left) eigenvector $\one_n$.  It remains to ensure that, with multiplicity, $\sum u_i$ is not one of the eigenvalues counted among $\lambda_1,\ldots,\lambda_{n-1}$. 

To see why this is the case, we consider the entries of $\vu$ as indeterminates so that $\det(\lambda I_n - \overline{L})$ can be viewed as a polynomial in $\RR[\lambda,u_1,\ldots,u_n]$.  For a generic choice of $u_1,\ldots,u_n$, it must be the case that $\lambda_j \neq \sum u_i$ for any $1 \leq j \leq n-1$.  Thus, $\sum u_i$ is an eigenvalue of $\overline{L}$ that is different from any of its other eigenvalues for generic $u_1,\ldots,u_n$. This means Equation~\eqref{eqn:u-update} holds when $\vu$ is generic.  Since both sides of Equation~\eqref{eqn:u-update} are polynomials in $\RR[\lambda,u_1,\ldots,u_n]$ that agree almost everywhere, they must be equal for all $\lambda, u_1,\ldots,u_n$.
\end{proof}

\begin{remark}
We proved a generalization of Lemma~\ref{KS-rank-one-result} for weighted graphs in \cite{Klee-Stamps-weighted}, which raises the question about a weighted version of Theorem~\ref{thm:main}.  Since its proof only relies on graph Laplacians being real symmetric matrices with rows (and columns) summing to zero, one can state Theorem~\ref{thm:main} in terms of weighted Laplacians without any additional work.  We omit the weighted version here, however, since the requirement of $\vv = \one_n$ in the rank-one matrix $\vu \vv^T$ in Theorem~\ref{thm:main} limits our ability to extend the applications in Section~\ref{section:applications} to their weighted analogs in a straightforward manner. 
\end{remark}

\section{Applications}\label{section:applications}

In this section, we demonstrate several ways Theorem~\ref{thm:main} can be applied to derive the characteristic polynomials of the Laplacian matrices of certain graphs.  The results we present here are not new, but the technique given by Theorem~\ref{thm:main} affords us direct and more straightforward proofs.

\subsection{Complete Graphs}

Let $K_n$ denote the \emph{complete graph} on $n$ vertices.

\begin{proposition}
The characteristic polynomial of $L(K_n)$ is $$\det(L(K_n) - \lambda I_n) = -\lambda (n-\lambda)^{n-1}.$$
\end{proposition}

\begin{proof}
Since $L(K_n) = nI_n - \one_n\one_n^T$, $$\det(L(K_n) + \one_n\one_n^T - \lambda I_n) = \det( nI_n - \lambda I_n) = (n-\lambda)^n.$$  The result follows from Theorem \ref{thm:main} by taking $\vu = \one_n$.
\end{proof}

\subsection{Complete Multipartite Graphs}

For $n_1,\ldots,n_p \in \mathbb{N}$, the \textit{complete multipartite graph} $K_{n_1,\ldots,n_p}$ is the graph whose vertex set is partitioned as $V_1 \cup \cdots \cup V_p$ with $|V_i| = n_i$ for all $i \in [p]$, no edges between any vertices in the same set $V_i$, and all possible edges between vertices in different sets $V_i$ and $V_j$.

\begin{proposition}\label{prop:multipartite}
Let $G$ be the complete multipartite graph $K_{n_1,\ldots,n_p}$ and let $n = n_1 + \cdots + n_p$.  The characteristic polynomial of $L(G)$ is 
$$
\det(L(G) - \lambda I_n) = -\lambda (n - \lambda)^{p-1} \prod_{i=1}^p (n-n_i - \lambda)^{n_i-1}.
$$
\end{proposition}

Before we prove Proposition~\ref{prop:multipartite}, let us first establish the following lemma.

\begin{lemma}\label{rank-one-update-identity}
For any $a,b \in \mathbb{R}$ and $n \in \mathbb{N}$, 
\begin{equation} 
\label{eigenvalues-identity-update}
\det\left(aI_n + b\one_n\one_n^T - \lambda I_n\right) = (a-\lambda)^{n-1}\left(a+bn - \lambda\right).
\end{equation}
\end{lemma}

\begin{proof}
First note that $(aI_n + b\one_n\one_n^T)\one_n = (a+bn)\one_n$, so $\one_n$ is an eigenvector of $aI_n + b\one_n\one_n^T$ with corresponding eigenvalue $a+bn$. If $\vv$ is any nonzero vector orthogonal to $\one_n$, then $(aI_n + b\one_n\one_n^T)\vv = a\vv$, so $\vv$ is an eigenvector of $aI_n + b\one_n\one_n^T$ with corresponding eigenvalue $a$.  

Therefore, the orthogonal complement to $\one_n$ is an $(n-1)$-dimensional eigenspace of $aI_n + b\one_n\one_n^T$ corresponding to the eigenvalue $\lambda = a$. So $\lambda = a$ has geometric (and hence algebraic) multiplicity $n-1$, while $\lambda = a+bn$ has multiplicity $1$.
\end{proof}

\begin{proof}[Proof of Proposition~\ref{prop:multipartite}.]
Order the vertices of $G$ so that the first $n_1$ vertices are those in $V_1$, the next $n_2$ vertices are those in $V_2$, and so on. Then $L(G) + \one_n\one_n^T$ is a block diagonal matrix whose diagonal blocks have the form $(n-n_i)I_{n_i} + \one_{n_i}\one_{n_i}^T$ for each $i \in [p]$. Therefore, by Lemma \ref{rank-one-update-identity}, 
\begin{align*}
\det\left((n-n_i)I_{n_i} + \one_{n_i}\one_{n_i}^T - \lambda I_{n_i} \right) &= 
\det\left((n-n_i - \lambda)I_{n_i} + \one_{n_i}\one_{n_i}^T\right) \\
&= (n-n_i - \lambda)^{n_i-1}(n - \lambda),
\end{align*} for each $i \in [p]$. It follows that $$\det(L + \one_n\one_n^T  - \lambda I_n) = (n-\lambda)^p \prod_{i=1}^p (n-n_i - \lambda)^{n_i-1},$$
and hence, by Theorem \ref{thm:main},
$$
\det(L-\lambda I_n) = -\lambda (n-\lambda)^{p-1} \prod_{i=1}^p (n-n_i - \lambda)^{n_i-1}.
$$
\end{proof}

\subsection{Complete Bipartite Graphs with a Perfect Matching Removed}

We turn our attention to graphs of the form $K_{n,n} \setminus M$ where $M$ is a perfect matching on $K_{n,n}$.  Figure~\ref{fig:bipartite-delete-matching} illustrates such a graph for $n=5$.

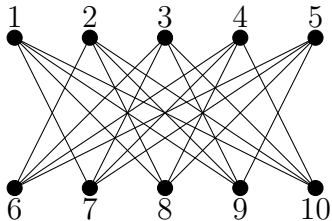
\begin{figure}[H]
    \centering
    \begin{tikzpicture}
    \foreach \x in {1,...,5}{
    \foreach \y in {1,...,5}{
    \ifthenelse{\x = \y}{}{\draw (\x,0) -- (\y,2);};
    }
    }
    \foreach \x in {1,...,5}{
    \draw[fill=black] (\x,2) circle (.1) node[above]{\x};
    }
    \foreach \x in {6,...,10}{
    \draw[fill=black] (\x-5,0) circle (.1) node[below]{\x};
    }
    
    \end{tikzpicture}
    \caption{The graph $K_{5,5} \setminus M$.}
    \label{fig:bipartite-delete-matching}
\end{figure}

\begin{proposition}
Let $G = K_{n,n} \setminus M$ where $M$ is a perfect matching on $K_{n,n}$.  The characteristic polynomial of $L(G)$ is 
$$
\det(L(G) - \lambda I_{2n}) = -\lambda \cdot (n-2 - \lambda)^{n-1} \cdot (n-\lambda)^{n-1} \cdot (2n-2 - \lambda).
$$
\end{proposition}

\begin{proof}
Order the vertices of $G = K_{n,n} \setminus M$ so that vertex $i$ is adjacent to vertex $n+i$ for $1 \leq i \leq n$. Consider $\overline{L} = L(G) + \one_{2n}\one_{2n}^T$, which has a block form 
$$
\overline{L} = 
\begin{bmatrix} 
A & I_n \\ 
I_n & A \\
\end{bmatrix},
$$
where $A = (n-1)I_n + \one_n \one_n^T$. Then, 
\begin{eqnarray*}
\det(\overline{L} - \lambda I_{2n} ) &=& 
\det\left(
\begin{bmatrix} 
 A-\lambda I_n & I_n \\ 
I_n & A - \lambda I_n \\
\end{bmatrix}
\right) \\
&\stackrel{(*)}{=}& \det((A - \lambda I_n)^2 - I_n ) \\
&=& \det(A - \lambda I_n - I_n)\det(A - \lambda I_n + I_n) \\
&=& \det((n-2 - \lambda)I_n + \one_n \one_n^T )\det((n-\lambda)I_n + \one_n \one_n^T ) \\
&\stackrel{(**)}{=}& (n-2 - \lambda)^{n-1}( 2n-2 - \lambda)(n - \lambda)^{n-1}(2n - \lambda).
\end{eqnarray*}
Note that (*) follows from the fact that $I_n$ commutes with $A - \lambda I_n$, and (**) follows from Lemma~\ref{rank-one-update-identity}.  The desired result follows from Theorem~\ref{thm:main}.
\end{proof}

\subsection{Threshold graphs} \label{section:threshold}

A nonempty graph is called \textit{threshold} if its vertices can be ordered such that each vertex  is adjacent to none or all of the vertices that come before it.  The former are called \textit{isolated} vertices and the latter are called \textit{dominating} vertices.  The first, or initial, vertex is neither isolated nor dominating.  If $G$ is a threshold graph, we use $I(G)$ and $D(G)$ respectively to denote the sets of isolated and dominating vertices in $G$.

Merris \cite{Merris} showed that the eigenvalues of a threshold graph $G$ on $n$ vertices are given by $$\lambda_i = \#\{v : \deg(v) \geq i\},$$ for $1 \leq i \leq n$. Later, Hammer and Kelmans \cite{Hammer-Kelmans} showed that the multiset of eigenvalues of $G$ can be equivalently expressed as $$\{0\} \cup \{\deg(v) : v \in I(G)\} \cup \{\deg(v) + 1 : v \in D(G)\}.$$ From this perspective, we see that the initial vertex is the only vertex whose degree does not contribute to the spectrum of $L(G)$. We conclude this paper with an elementary proof for the description of the spectrum of the Laplacian matrix of a threshold graph in terms of its vertex degrees.  Specifically, we show that with a convenient choice of rank-one perturbation, there is a matrix $\overline{L}(G)$ for which every vertex degree contributes to the spectrum; and moreover that the degree of the initial vertex can be replaced with the eigenvalue $\lambda = 0$ to obtain the spectrum of $L(G)$.

\begin{proposition}
Let $G$ be a threshold graph on $n$ vertices.  The characteristic polynomial of $L(G)$ is 
$$\det(L(G) - \lambda I_n) = -\lambda \prod_{j \in I(G)} (\deg(j) - \lambda) \prod_{j \in D(G)} (\deg(j)+1 - \lambda).
$$
\end{proposition}

\begin{proof}
Order the vertices of $G$ according to their positions in the isolated-dominating construction sequence and let $\vu \in \mathbb{R}^n$ be the indicator vector of $D(G)$.  Go et al. \cite{GKLS} note that for $i<j$ in this prescribed ordering, $\{i,j\}$ is an edge of $G$ if vertex $j$ is dominating and it is not an edge if vertex $j$ is isolated.  Thus,  $\overline{L}(G) = L(G) + \vu \one_n^T$ is upper triangular with diagonal entries $\deg(j) + 1$ for each $j \in D(G)$ and $\deg(j)$ for each $j \in I(G) \cup \{1\}$, which means the characteristic polynomial of $\overline{L}(G)$ is $$\det( \overline{L}(G) - \lambda I_n) = (\deg(1) - \lambda) \prod_{j \in I(G)} (\deg(j) - \lambda) \prod_{j \in D(G)} (\deg(j)+1 - \lambda).$$  The desired result follows from Theorem~\ref{thm:main} by observing that the neighborhood of the initial vertex is precisely the set of dominating vertices in $G$, which implies $\sum_{i=1}^n u_i = \deg(1)$.  
\end{proof}
\section*{Acknowledgments}

We are grateful to Mohamed Omar for helpful and encouraging conversations during the early stages of this project.

\bibliographystyle{plain}

\end{document}